\documentclass{amsart}
\usepackage{amsmath,amsthm,hyperref,amsfonts,amssymb}
\hypersetup{
    colorlinks=true,       
    linkcolor=blue,          
    citecolor=cyan,        
}
\usepackage{verbatim}
\usepackage{url}

\makeatother

\newtheorem{theorem}{Theorem}[section]
\newtheorem{lemma}[theorem]{Lemma}
\newtheorem{corollary}[theorem]{Corollary}
\newtheorem{proposition}[theorem]{Proposition}
\newtheorem{definition}{Definition}[section]
\newtheorem{conjecture}{Conjecture}
\theoremstyle{remark}
\newtheorem{remark}[definition]{Remark}
\usepackage{tikz}
\usetikzlibrary{decorations.markings}
\makeatother

\def\P{\partial}

\def\ri{\mathrm i}

\def\qb{\mathbb Q}
\def\rb{\mathbb R}

\def\zb{\mathbb Z}
\def\cb{{\mathbb C}}

\def\rrw{\rightarrow}
\numberwithin{equation}{section}

\title{Proof of a conjecture of Farkas and Kra}
\author{Nian Hong Zhou}
\address{School of Mathematical Sciences,  East China Normal University,
500 Dongchuan Road, Shanghai 200241, PR China}
\email{nianhongzhou@outlook.com}

\thanks{This research was supported by the National Science Foundation of China (Grant No. 11571114).}
\begin{document}
\begin{abstract}
In this paper we prove a conjectured modular equation of Farkas and Kra, which involving a half sum of certain modular form of weight $1$ for congruence subgroup $\Gamma_1(k)$ with any prime $k$. We prove that their conjectured identity holds for all odd integer $k\ge 2$. A new modular equation of Farkas and Kra type is also established.
\end{abstract}
\maketitle
\section{Introduction and statement of results}
In this paper, we let $z\in \cb$, $\tau\in\cb$ with $\Im(\tau)>0$ and $q=e^{2\pi\ri \tau}$. The theta function with
characteristic $\left[\begin{matrix}\epsilon\\ \epsilon'\end{matrix}\right]\in\rb^2$ is defined by
\begin{equation}\label{deft}
\theta\left[\begin{matrix}\epsilon\\ \epsilon'\end{matrix}\right](z,\tau)=\sum_{n\in\zb}\exp\left\{2\pi\ri\left(\frac{1}{2}\left(n+\frac{\epsilon}{2}\right)^2\tau+\left(n+\frac{\epsilon}{2}\right)\left(z+\frac{\epsilon'}{2}\right)\right)\right\},
\end{equation}
which is a generalization of the Jacobi theta functions. The theory of above theta function was systematically studied by Farkas and Kra \cite{MR1850752}, which play an important role in combinatorial number theory, algebraic geometry and physics.

In \cite[Chapter 4]{MR1850752}, Farkas and Kra treated the theta function \eqref{deft} with $\epsilon,\epsilon'\in \qb$ and $z=0$, that is, the theta constants with rational characteristics. Their  derived many interesting results, one of them is the following (see \cite[Theorem 9.8, p.318]{MR1850752} and \cite{MR1940167}):

\begin{theorem}
For each odd prime $k$ and all $\tau\in\cb$ with $\Im(\tau)>0$,
\begin{equation}\label{beeq}
\frac{\,d}{\,d\tau}\log\left(\frac{\eta(k\tau)}{\eta(\tau)}\right)+\frac{1}{2\pi\ri(k-2)}\sum_{0\le \ell\le \frac{k-3}{2}}\left(\frac{\theta'\left[\begin{matrix}1\\ \frac{1+2\ell}{k}\end{matrix}\right](0,\tau)}{\theta\left[\begin{matrix}1\\ \frac{1+2\ell}{k}\end{matrix}\right](0,\tau)}\right)^2
\end{equation}
is a cusp $1$-form (cusp form of weight $1$) for the Hecke congruence subgroup $\Gamma_o(k)$. This form is identically zero provided $k\le 19$.  Here $\eta(\tau)=q^{1/24}\prod_{n\ge 1}(1-q^n)$ is the Dedekind eta function and
$$\theta'\left[\begin{matrix}1\\ \frac{1+2\ell}{k}\end{matrix}\right](0,\tau)=\frac{\P}{\P z} \theta\left[\begin{matrix}1\\ \frac{1+2\ell}{k}\end{matrix}\right](z,\tau)\bigg|_{z=0}.$$
\end{theorem}
They then in \cite[Conjecture 9.10, p.320]{MR1850752} (see also \cite{MR1940167}) conjectured that \eqref{beeq} is identically zero for each odd prime $k$ and all $\tau\in\cb$ with $\Im(\tau)>0$.

\begin{remark}
We remark that for odd integers $k, \ell$ with $k\ge 3$,
$$\left[\frac{\P}{\P z}\log\left(\theta\left[\begin{matrix}1\\ \ell/k\end{matrix}\right](0,\tau)\right)\right]^2$$ is a modular $1$-form (modular form of wight $1$) for the group:
$$G(k)=\Gamma_1(k):=\left\{\left(\begin{matrix}a&b\\ c&d\end{matrix}\right)\in{\rm SL}_2(\zb): \left(\begin{matrix}a&b\\ c&d\end{matrix}\right)\equiv \left(\begin{matrix}1&\ast\\ 0&1\end{matrix}\right)~(\bmod ~k)\right\}.$$
This fact and more related results can be found in \cite{MR1850752, MR1940167}.
\end{remark}

The aim of this paper is give a proof of the conjecture of Farkas and Kra of above. For the simplicity of the proof, we shall introduce the Jacobi theta function $\theta_2(z,q)$, which is defined by (see for example \cite{MR0178117}):
\begin{equation}\label{defeq}
\theta_{2}(z,q)=\sum_{n\in\zb}q^{(2n+1)^2/8}e^{\ri(2n+1)z}.
\end{equation}
Hence it is clear that
$$
\theta\left[\begin{matrix}1\\ \epsilon'\end{matrix}\right](z,\tau)= \theta_2\left(\pi z+\frac{\epsilon'\pi}{2}, q\right)
$$
and the conjecture of we concerned is equivalent to the following.
\begin{conjecture}\label{coj1}
For each odd prime $k$ and all $\tau\in\cb$ with $\Im(\tau)>0$,
\begin{equation*}
4(k-2)q\frac{\,d}{\,d q}\log\left(\frac{\eta(k\tau)}{\eta(\tau)}\right)-\sum_{\substack{0\le \ell<k\\ \ell \equiv 1~(\bmod 2)}}\left[\frac{\P}{\P z}\log\theta_2\left(\frac{\ell}{2k}\pi,q\right)\right]^2=0.
\end{equation*}
\end{conjecture}

We shall prove a more general result than Conjecture \ref{coj1}. To statement our main result, we shall consider the following half sum:
\begin{equation}\label{coeq}
S_\delta(k):=\sum_{\substack{0\le \ell<k\\ \ell-k\equiv \delta~(\bmod 2)}}\left[\frac{\P}{\P z}\log\theta_2\left(\frac{\ell}{2k}\pi,q\right)\right]^2
\end{equation}
for each integer $k\ge 2$ and each $\delta\in\{0,1\}$. Our main result is the following two modular equations.
\begin{theorem}\label{mt} For all $\tau\in\cb$ with $\Im(\tau)>0$, we have if $\delta=0$ then
$$
S_\delta(k)=4(k-2)q\frac{\,d}{\,dq}\log\left(\frac{\eta(k\tau)}{\eta(\tau)}\right),
$$
and if $\delta=1$ then
$$
S_\delta(k)=4q\frac{\,d }{\,d q}\log\left(\frac{\eta(2k\tau)^{2k-2}}{\eta(\tau)^k\eta(k\tau)^{k-2}}\right).
$$
\end{theorem}
We immediately obtain the proof of Conjecture \ref{coj1} by setting $k\in 2\zb_++1$ and $\delta=0$ in Theorem \ref{mt}.
\begin{corollary}
Conjecture \ref{coj1} holds for all odd integer $k\ge 2$. In particular, Conjecture \ref{coj1} is true.
\end{corollary}
We shall give some consequence of Theorem \ref{mt}. For this purpose we first use Lemma \ref{lemd} of below deduce the proposition as follows.
\begin{proposition}\label{pro} We have
$$S_{\delta}(k)=\sum_{\substack{0\le \ell<k\\ \ell-k\equiv \delta~(\bmod 2)}}\left[\tan\left(\frac{\ell\pi}{2k}\right)-4\sum_{h=1}^{2k}(-1)^h\sin\left(\frac{\ell h\pi}{k}\right)\sum_{n\ge 1}\frac{q^{hn}}{1-q^{2kn}}\right]^2.$$
\end{proposition}
By setting $q=0$ in Theorem \ref{mt}, application Proposition \ref{pro} and \eqref{etad} of below we obtain the following trigonometric identity, which has been appeared in \cite{MR1850752, MR1940167}.
\begin{corollary}For each integer $k\ge 2$,
$$\sum_{\substack{0\le \ell<k\\ \ell-k\equiv \delta~(\bmod 2)}}\left[\tan\left(\frac{\ell\pi}{2k}\right)\right]^2=\begin{cases}\frac{(k-1)(k-2)}{6}\quad & {\text if}~\delta=0,\\
\quad \frac{k(k-1)}{6}& {\text if}~\delta=1.
\end{cases} $$
\end{corollary}
From Theorem \ref{mt}, Proposition \ref{pro} and \eqref{etad}, by choose different pair $(k, \delta)$ one can obtain many Lambert series identities.  For example, if we pick $(k,\delta)=(3,1)$, then it is easy to see that:
\begin{corollary} We have
$$\left(1+2\sum_{n\ge 1}\frac{q^n+q^{2n}-q^{4n}-q^{5n}}{1-q^{6n}}\right)^2=1+4\sum_{n\ge 1}\left(\frac{nq^{n}}{1-q^{n}}+\frac{nq^{3n}}{1-q^{3n}}-\frac{8nq^{6n}}{1-q^{6n}}\right).$$
\end{corollary}

\section{Primaries}\label{sec2}
We shall need the following primary results, which will be used to prove main results of this paper.
\begin{proposition}\label{meq1}We have:
\begin{equation*}
\left(\frac{\P}{\P z}\log \theta_{2}(z,q)\right)^2={\rm T}_{z,q}\left(\log \theta_{2}(z,q)\right),
\end{equation*}
where and throughout, ${\rm T}_{z,q}$ is a linear operator be defined as
$${\rm T}_{z,q}=-8q\frac{\P }{\P q}-\frac{\P^2}{\P z^2}.$$
\end{proposition}
\begin{proof}
By \eqref{defeq} it is clear that
$$
\left(8q\frac{\P}{\P q} +\frac{\P^2 }{\P z^2}\right)\theta_{2}(z,q)=0,
$$
which means that
$$\frac{1}{\theta_{2}(z,q)}\frac{\P^2}{\P z^2}\theta_{2}(z,q)=-8q\frac{\P }{\P q}\log \theta_{2}(z,q).$$
Then from the basic fact that
$$
\frac{\P^2}{\P z^2}\log \theta_{2}(z,q)= \frac{1}{\theta_{2}(z,q)}\frac{\P^2}{\P z^2}\theta_{2}(z,q)-\left(\frac{\P}{\P z}\log \theta_{2}(z,q)\right)^2
$$
we complete the proof of the proposition.
\end{proof}

We need the Jacobi triple product identity for $\theta_2(z,q)$ (see for example \cite{MR0171725, MR0178117}),
\begin{equation}\label{treq}
\theta_{2}(z,q)=q^{1/8}e^{-\ri z}\prod_{n\ge 1}(1-q^n)\left(1+e^{-2\ri z}q^{n}\right)\left(1+e^{2\ri z}q^{n-1}\right).
\end{equation}

\begin{lemma}\label{lemd}For each $\ell, k\in\zb$ with $\ell\neq k$ and $k>0$,
$$-\frac{\P}{\P z}\log \theta_2\left(\frac{\ell}{2k}\pi,q\right)=\tan\left(\frac{\ell\pi}{2k}\right)-4\sum_{h=1}^{2k}(-1)^h\sin\left(\frac{\ell h\pi}{k}\right)\sum_{n\ge 1}\frac{q^{hn}}{1-q^{2kn}}.$$
\end{lemma}
\begin{proof}

Taking the logarithmic derivative of $\theta_2(z,q)$ respect to $z$ by \eqref{treq}, we have the well known
Fourier expansion:
\begin{equation}\label{fourt2}
\frac{\P}{\P z}\log \theta_2(z,q)=-\tan(z)+4\sum_{n\ge 1}\frac{(-1)^nq^n}{1-q^n}\sin(2nz).
\end{equation}
Notice that
\begin{align*}
\sum_{n\ge 1}\frac{(-1)^nq^n}{1-q^n}\sin\left(2n\frac{\ell \pi}{2k}\right) &=\sum_{h=1}^{2k}  \sum_{n\ge 0}\frac{(-1)^{h}q^{2nk+h}}{1-q^{2nk+h}}\sin\left(\frac{\ell h\pi}{k}\right) \\
&=\sum_{h=1}^{2k}(-1)^h\sin\left(\frac{\ell h\pi}{k}\right)\sum_{n\ge 0}\sum_{\ell\ge 1}q^{(2nk+h)\ell}
\end{align*}
and \eqref{fourt2} we immediately obtain that
$$\frac{\P}{\P z}\log \theta_2\left(\frac{\ell}{2k}\pi,q\right)=-\tan\left(\frac{\ell\pi}{2k}\right)+4\sum_{h=1}^{2k}(-1)^h\sin\left(\frac{\ell h\pi}{k}\right)\sum_{n\ge 1}\frac{q^{hn}}{1-q^{2kn}}.$$
This completes the proof of the lemma.
\end{proof}
The following lemma will be used to proof Theorem \ref{mt} in next section.
\begin{lemma}\label{lem22}We have:
\begin{equation*}
{\rm T}_{z,q}\left(\log \theta_2\left(kz,q^k\right)\right)\bigg|_{z=0}=8(k-1)q\frac{\,d}{\,dq}\log\left(\frac{\eta(2k\tau)^2}{\eta(k\tau)}\right)
\end{equation*}
and
\begin{equation*}
{\rm T}_{z,q}\left(\log \left(\frac{\theta_2\left(kz-\frac{\pi}{2},q^k\right)}{\theta_2\left(z-\frac{\pi}{2},q\right)}\right)\right)\bigg|_{z=0}=8q\frac{\,d}{\,dq}\log\left(\eta(k\tau)^{k-3}\eta(\tau)^2\right).
\end{equation*}
\end{lemma}
\begin{proof}
By \eqref{fourt2} we have: \begin{equation*}
\frac{\P^2}{\P z^2}\log \theta_2(z,q)=-\tan^2(z)-1+8\sum_{n\ge 1}\frac{(-1)^nnq^n}{1-q^n}\cos(2nz)
\end{equation*}
and
\begin{equation*}
\frac{\P^2}{\P z^2}\log \theta_2(z-\pi/2,q)=-\cot^2(z)-1+8\sum_{n\ge 1}\frac{nq^n}{1-q^n}\cos(2nz).
\end{equation*}
Hence we obtain that
\begin{align*}
\frac{\P^2}{\P z^2}\log \theta_2(z,q)\bigg|_{z=0}=&-1+8\sum_{n\ge 1}\frac{(-1)^nnq^n}{1-q^n}\\
=&-1+16\sum_{n\ge 1}\frac{2nq^{2n}}{1-q^{2n}}-8\sum_{n\ge 1}\frac{nq^n}{1-q^n}
\end{align*}
and
\begin{align*}
\frac{\P^2}{\P z^2}&\log \left(\frac{\theta_2\left(kz-\frac{\pi}{2},q^k\right)}{\theta_2\left(z-\frac{\pi}{2},q\right)}\right)\bigg|_{z=0}\\
&=\lim_{z\rrw 0}\left(\cot^2(z)+1-k^2(\cot^2(kz)+1)\right)+8\sum_{n\ge 1}\left(\frac{k^2nq^{kn}}{1-q^{kn}}-\frac{nq^n}{1-q^n}\right)\\
&=\frac{1-k^2}{3}+8k^2\sum_{n\ge 1}\frac{nq^{kn}}{1-q^{kn}}-8\sum_{n\ge 1}\frac{nq^n}{1-q^n}.
\end{align*}
By using of the fact that
\begin{equation}\label{etad}
q\frac{\,d}{\,d q}\log\eta(\alpha\tau)=\frac{\alpha}{24}-\sum_{n\ge 1}\frac{\alpha nq^{\alpha n}}{1-q^{\alpha n}}, ~~~\alpha\in\rb_+,
\end{equation}
and the above we obtain
\begin{equation}\label{2dt2}
\frac{\P^2}{\P z^2}\log \theta_2(z,q)\bigg|_{z=0}=8q\frac{\,d}{\,d q}\log\left(\frac{\eta(\tau)}{\eta(2\tau)^2}\right)
\end{equation}
and
\begin{equation}\label{2dt1}
\frac{\P^2}{\P z^2}\log \left(\frac{\theta_2\left(kz-\frac{\pi}{2},q^k\right)}{\theta_2\left(z-\frac{\pi}{2},q\right)}\right)\bigg|_{z=0}=8q\frac{\,d}{\,d q}\log\left(\frac{\eta(\tau)}{\eta(k\tau)^k}\right).
\end{equation}
Moreover, by \eqref{treq} and the definition of $\eta(\tau)$, it is easy to see that
\begin{equation}\label{t0}
\theta_2(0,q)=2\frac{\eta(2\tau)^2}{\eta(\tau)}
\end{equation}
and
\begin{equation}\label{t1}
\lim_{z\rrw 0}\frac{\theta_2\left(z-{\pi}/{2},q\right)}{z}=2\eta(\tau)^3.
\end{equation}
Thus for integer $k\ge 1$, application of \eqref{2dt2} and \eqref{t0} implies that
\begin{align*}
{\rm T}_{z,q}&\left(\log \theta_2\left(kz,q^k\right)\right)\bigg|_{z=0}\\
&=-8q\frac{\,d}{\,dq}\log\theta_2(0,q^k)-\frac{\P^2}{\P z^2}\log \theta_2(kz,q^k)\bigg|_{z=0}\\
&=-8q\frac{\,d}{\,dq}\log\left(\frac{\eta(2k\tau)^2}{\eta(k\tau)}\right)+k^2\left(-8q^k\frac{\,d}{\,d q^k}\log\left(\frac{\eta(k\tau)}{\eta(2k\tau)^2}\right)\right)\\
&=8(k-1)q\frac{\,d}{\,dq}\log\left(\frac{\eta(2k\tau)^2}{\eta(k\tau)}\right),
\end{align*}
and application of \eqref{2dt1} and \eqref{t1} implies that
\begin{align*}
{\rm T}_{z,q}&\left(\log \left(\frac{\theta_2\left(kz-\frac{\pi}{2},q^k\right)}{\theta_2\left(z-\frac{\pi}{2},q\right)}\right)\right)\bigg|_{z=0}\\
&=-8q\frac{\,d}{\,dq}\log \left(\frac{\eta(k\tau)^3}{\eta(\tau)^3}\right)-\frac{\P^2}{\P z^2}\log \left(\frac{\theta_2\left(kz-\frac{\pi}{2},q^k\right)}{\theta_2\left(z-\frac{\pi}{2},q\right)}\right)\bigg|_{z=0}\\
&=-24q\frac{\,d}{\,dq}\log \left(\frac{\eta(k\tau)}{\eta(\tau)}\right)-8q\frac{\,d}{\,d q}\log\left(\frac{\eta(\tau)}{\eta(k\tau)^k}\right)\\
&=8q\frac{\,d}{\,dq}\log\left(\eta(k\tau)^{k-3}\eta(\tau)^2\right),
\end{align*}
which completes the proof of the lemma.
\end{proof}
We need the following half product formula for Jacobi theta function $\theta_2$, which will be used to proof Theorem \ref{mt} in next section.
\begin{lemma}\label{lem2}For integer $k\ge 1$ and $\delta\in\{0,1\}$,
\begin{equation*}
\prod_{\substack{0\le \ell< 2k\\ \ell-k\equiv \delta~(\bmod 2)}}\theta_2\left(z+\frac{\ell}{2k}\pi,q\right)=C_{k,\delta}\frac{\eta(\tau)^k}{\eta(k\tau)}\theta_2\left(kz+\frac{(\delta-1)\pi}{2},q^k\right),
\end{equation*}
where $C_{k,\delta}=e^{\frac{\ri\pi}{2}(\delta-k+{\bf 1}_{k\not\equiv \delta~(\bmod 2)})}$. Here and throughout, ${\bf 1}_{condition}=1$ if the 'condition' is true, and equals to $0$ if the 'condition' is false.
\end{lemma}
\begin{proof}From \eqref{treq} we have
\begin{align*}
\prod_{\substack{0\le \ell< 2k\\ \ell-k\equiv \delta~(\bmod 2)}}&\theta_2\left(z+\frac{\ell\pi}{2k},q\right)\\
=&\prod_{\substack{0\le \ell<2k\\ \ell-k\equiv \delta~(\bmod 2)}}\left(q^{1/8}e^{-\ri (z+\frac{\ell}{2k}\pi)}\prod_{n\ge 1}(1-q^n)\right)\\
&\times\prod_{n\ge 1}\prod_{\substack{0\le \ell< 2k\\ \ell-k\equiv \delta~(\bmod 2)}}\left(1+ q^{n}e^{-2\ri z-\frac{\ell \pi\ri}{k}}\right)\left(1+q^{n-1}e^{2\ri z+\frac{\ell \pi\ri}{k}}\right).
\end{align*}
It is easy to check that
$$\prod_{\substack{0\le \ell<2k\\ \ell-k\equiv \delta~(\bmod 2)}}(1+x e^{\pm \frac{\ell \pi\ri}{k}})=1-e^{\delta \pi\ri}x^k$$
and
$$\sum_{\substack{0\le \ell<2k\\ \ell-k\equiv \delta~(\bmod 2)}}\ell =k(k-1)+k{\bf 1}_{k\not\equiv \delta~(\bmod 2)}.$$
Thus we obtain that
\begin{align*}
\prod_{\substack{0\le \ell< 2k\\ \ell-k\equiv \delta~(\bmod 2)}}&\theta_2\left(z+\frac{\ell}{2k}\pi,q\right)\\
=&q^{k/12}\eta(\tau)^ke^{-\ri k z}e^{-\frac{\ri \pi}{2}\left(k-1+{\bf 1}_{k\not\equiv \delta~(\bmod 2)}\right)}\\
&\times\prod_{n\ge 1}\left(1-e^{-2\ri kz-\delta\pi\ri}q^{kn}\right)\left(1-e^{2\ri kz+\delta\pi\ri}q^{k(n-1)}\right)\\
=&C_{k,\delta}\theta_2\left(kz+(\delta-1)\pi/2,q^k\right)\frac{\eta(\tau)^k}{\eta(k\tau)},
\end{align*}
with
$$
C_{k,\delta}=e^{\frac{\ri\pi(\delta-1)}{2}-\frac{\ri \pi\left(k-1+{\bf 1}_{k\not\equiv \delta~(\bmod 2)}\right)}{2}}=e^{\frac{\ri\pi}{2}(\delta-k+{\bf 1}_{k\not\equiv \delta~(\bmod 2)})},
$$
which completes the proof of the lemma.
\end{proof}
\section{The proof of Theorem \ref{mt}}
First of all, we shall define
\begin{align*}
G_{\delta,k}(z,q)&:=\sum_{\substack{0\le \ell<k\\ \ell-k\equiv \delta~(\bmod 2)}}\left[\frac{\P}{\P z}\log\theta_2\left(z+\frac{\ell}{2k}\pi,q\right)\right]^2,
\end{align*}
then from \eqref{coeq} we have $S_{\delta}(k)=G_{\delta,k}(0,q)$. By Proposition \ref{meq1} we get
\begin{align*}
G_{\delta,k}(z,q)
&=\sum_{\substack{0\le \ell<k\\ \ell-k\equiv \delta~(\bmod 2)}}{\rm T}_{z,q}\left(\log\theta_2\left(z+\frac{\ell}{2k}\pi,q\right)\right)\\
&={\rm T}_{z,q}\left(\sum_{\substack{0\le \ell<k\\ \ell-k\equiv \delta~(\bmod 2)}}\log \theta_2\left(z+\frac{\ell}{2k}\pi,q\right)\right).
\end{align*}
We shall define the auxiliary function as follows:
$$A_{\delta,k}(z,q)={\rm T}_{z,q}\left(\sum_{\substack{0\le \ell\le 2k\\ \ell-k\equiv \delta~(\bmod 2)}}\log \theta_2\left(z+\frac{\ell\pi}{2k},q\right)-{\bf 1}_{\delta=0}\log \theta_2\left(z+\frac{\pi}{2},q\right)\right).$$
We claim that
\begin{equation}\label{eqmm1}
G_{\delta,k}(0,q)=\frac{1}{2}\lim_{z\rrw 0}A_{\delta,k}(z,q).
\end{equation}
In fact, by $\theta_2(z+\pi,q)=-\theta_2(z,q)$ and $\theta_2(z,q)=\theta_2(-z,q)$ we have
\begin{align*}
A_{\delta,k}(z,q)=&\sum_{\substack{0\le \ell<k\\ \ell-k\equiv \delta~(\bmod 2)}}{\rm T}_{z,q}\left(\log \theta_2\left(z+\frac{\ell\pi}{2k},q\right)+\log \theta_2\left(z+\frac{(2k-\ell)\pi}{2k},q\right)\right)\\
=&\sum_{\substack{0\le \ell<k\\ \ell-k\equiv \delta~(\bmod 2)}}{\rm T}_{z,q}\left(\log \theta_2\left(z+\frac{\ell}{2k}\pi,q\right)+\log \theta_2\left(-z+\frac{\ell}{2k}\pi,q\right)\right).
\end{align*}
By the definition of ${\rm T}_{z,q}$ and above we immediately obtain the proof of \eqref{eqmm1}. From Lemma \ref{lem2} we find that
\begin{align*}
\sum_{\substack{0\le \ell< 2k\\ \ell-k\equiv \delta~(\bmod 2)}}\log \theta_2\left(z+\frac{\ell\pi}{2k},q\right)&=\log \prod_{\substack{0\le \ell< 2k\\ \ell-k\equiv \delta~(\bmod 2)}}\theta_2\left(z+\frac{\ell}{2k}\pi,q\right)\\
&=\log\left(C_{k,\delta}\frac{\eta(\tau)^k}{\eta(k\tau)}\theta_2\left(kz+\frac{(\delta-1)\pi}{2},q^k\right)\right),
\end{align*}
which implies that
\begin{align*}
A_{\delta,k}(z,q)=&{\rm T}_{z,q}\left(\log \theta_2\left(kz+\frac{(\delta-1)\pi}{2},q^k\right)\right)-8q\frac{\P }{\P q}\log\left(\frac{\eta(\tau)^k}{\eta(k\tau)}\right)\\
&+{\rm T}_{z,q}\left({\bf 1}_{k\equiv \delta~(\bmod 2)}\log \theta_2\left(z,q\right)-{\bf 1}_{\delta=0}\log \theta_2\left(z+\frac{\pi}{2},q\right)\right).
\end{align*}
Further by Lemma \ref{lem22} we obtain that
\begin{align*}
A_{0,k}(0,q)=&-8q\frac{\P }{\P q}\log\left(\frac{\eta(\tau)^k}{\eta(k\tau)}\right)+{\rm T}_{z,q}\left(\log \left(\frac{\theta_2\left(kz-\frac{\pi}{2},q^k\right)}{\theta_2\left(z-\frac{\pi}{2},q\right)}\right)\right)\bigg|_{z=0}\\
=&8(k-3)q\frac{\,d}{\,dq}\log\eta(k\tau)+16q\frac{\,d}{\,dq}\log\eta(\tau)-8q\frac{\,d}{\,d q}\log\left(\frac{\eta(\tau)^k}{\eta(k\tau)}\right)\\
=&8(k-2)q\frac{\,d}{\,dq}\log\left(\frac{\eta(k\tau)}{\eta(\tau)}\right)
\end{align*}
and
\begin{align*}
A_{1,k}(0,q)=&-8q\frac{\P }{\P q}\log\left(\frac{\eta(\tau)^k}{\eta(k\tau)}\right)+{\rm T}_{z,q}\left(\log \theta_2\left(kz,q^k\right)\right)\bigg|_{z=0}\\
=&-8q\frac{\,d }{\,d q}\log\left(\frac{\eta(\tau)^k}{\eta(k\tau)}\right)+8(k-1)q\frac{\,d}{\,dq}\log\left(\frac{\eta(2k\tau)^2}{\eta(k\tau)}\right)\\
=&8q\frac{\,d }{\,d q}\log\left(\frac{\eta(2k\tau)^{2k-2}}{\eta(\tau)^k\eta(k\tau)^{k-2}}\right),
\end{align*}
which completes the proof of Theorem \ref{mt} by note that $S_{\delta}(k)=G_{\delta,k}(0,q)=\frac{1}{2}A_{1,k}(0,q)$.

\section*{Acknowledgment}
The author would like to thank his advisor Zhi-Guo Liu for consistent encouragement and useful suggestions.


\begin{thebibliography}{1}

\bibitem{MR1850752}
Hershel~M. Farkas and Irwin Kra.
\newblock {\em Theta constants, {R}iemann surfaces and the modular group},
  volume~37 of {\em Graduate Studies in Mathematics}.
\newblock American Mathematical Society, Providence, RI, 2001.
\newblock An introduction with applications to uniformization theorems,
  partition identities and combinatorial number theory.

\bibitem{MR1940167}
Hershel~M. Farkas and Irwin Kra.
\newblock On theta constant identities and the evaluation of trigonometric
  sums.
\newblock In {\em Complex manifolds and hyperbolic geometry ({G}uanajuato,
  2001)}, volume 311 of {\em Contemp. Math.}, pages 115--131. Amer. Math. Soc.,
  Providence, RI, 2002.

\bibitem{MR0178117}
E.~T. Whittaker and G.~N. Watson.
\newblock {\em A course of modern analysis. {A}n introduction to the general
  theory of infinite processes and of analytic functions: with an account of
  the principal transcendental functions}.
\newblock Fourth edition. Reprinted. Cambridge University Press, New York,
  1962.

\bibitem{MR0171725}
George~E. Andrews.
\newblock A simple proof of {J}acobi's triple product identity.
\newblock {\em Proc. Amer. Math. Soc.}, 16:333--334, 1965.

\end{thebibliography}
\end{document}